\documentclass[a4paper,12pt]{article}
\usepackage{latexsym,amsmath,amssymb,amsthm}
\usepackage{verbatim}
\usepackage{color}
\usepackage[left=2cm,right=1.5cm,top=2cm,bottom=3cm]{geometry}
\theoremstyle{plain}
\newtheorem{theorem}{Theorem}[section]

\newtheorem{proposition}[theorem]{Proposition}
\newtheorem{corollary}[theorem]{Corollary}
\newtheorem{definition}[theorem]{Definition}

\date{}
\title{Some remarks on ``Convergence of Picard's iteration using projection algorithm for noncyclic contractions" [Indag. Math. 30 (2019) 227--239]}
\author{Moosa Gabeleh\footnote{Department of Mathematics, Ayatollah Boroujerdi
University, Boroujerd, Iran; email:
Gabeleh@abru.ac.ir, gab.moo@gmail.com,},  Hans-Peter A. K\"unzi\footnote{Department of Mathematics and Applied Mathematics, University of Cape Town, Rondebosch 7701, South Africa; email: hans-peter.kunzi@uct.ac.za}}

\begin{document}
\maketitle

\noindent{\bf Abstract.}
In this note, at first we prove that the existence of best proximity points for cyclic relatively nonexpansive mappings is equivalent to the
existence of best proximity pairs for noncyclic relatively nonexpansive mappings in the setting of strictly convex Banach spaces by using
the projection operator. In this way, we conclude that a main result of the paper "Proximal normal
structure and relatively nonexpansive mappings", Studia Math.,
(171~(2005) 283--293) immediately follows. We then
discuss the convergence of best proximity pairs for noncyclic contractions by applying the convergence of iterative sequences
for cyclic contractions and show that the convergence method of a recent paper published in Indag. Math., 30(1) (2019) 227--239
is obtained exactly from  Picard's iteration sequence.
\\

  \maketitle\noindent {\bf Key words}:
Best proximity (point) pair; uniformly convex Banach space; noncyclic (cyclic) contraction.
\maketitle

\section{Introduction}

Throughout this paper $(A,B)$ is a pair of
nonempty and disjoint subsets of a normed linear space $X$. A mapping $T : A \cup
B \rightarrow A \cup B$ is said to be \emph{cyclic} if
$T(A)\subseteq B$ and $T(B)\subseteq A$. Also, $T$ is called a \emph{noncyclic mapping} $T(A)\subseteq A$ and $T(B)\subseteq B$.

\begin{definition}
A point $(p,q)\in A\times B$ is said to be a best proximity pair for the noncyclic mapping $T:A \cup B \rightarrow A \cup B$
provided that
$$
p=Tp, \quad q=Tq \quad \text{and} \quad d(p,q)={\rm dist}(A,B):=\inf\{d(x,y) : (x,y)\in A\times B\}.
$$
Also, if $T : A \cup B \rightarrow A \cup B$ is a cyclic mapping, then a point $x^\star\in A\cup B$ is called a best proximity point for $T$
provided that
$$
\|x^\star-Tx^\star\|={\rm dist}(A,B).
$$
\end{definition}

\begin{definition}
$T : A \cup B \rightarrow A \cup B$ is called a cyclic (noncyclic) relatively nonexpansive mapping, if $T$ is cyclic (noncyclic) and
$$
\|Tx-Ty\|\leq \|x-y\|, \quad \forall (x,y)\in A\times B.
$$
Also, $T$ is said to be a cyclic (noncyclic) contraction provided that $T$ is cyclic (noncyclic) and there exists $\alpha\in (0,1)$
for which
$$
d(Tx,Ty)\leq \alpha d(x,y)+ (1-\alpha) \rm{dist}(A,B),
$$
for all $(x,y)\in A\times B$.
\end{definition}

It is clear that the class of cyclic (noncyclic) relatively nonexpansive mappings contains the class of cyclic (noncyclic) contractions
as a subclass.
\par
In order to state the existence and convergence results of best proximity points (pairs) we need to recall the following concepts and
notations.

\begin{definition}
A Banach space $X$ is said to be\\
$(i)$ uniformly convex if there exists a strictly increasing function $\delta : [0,2]\to [0,1]$
such that the following implication holds for all $x,y,p\in X , R>0$ and $r\in[0,2R]$ :
$$
\begin{cases} \|x-p\|\leq R,\\
\|y-p\|\leq R,\\
\|x-y\|\geq r
\end{cases}\Rightarrow \|\frac{x+y}{2}-p\|\leq(1-\delta(\frac{r}{R}))R;
$$
$(ii)$ strictly convex if the following implication holds for all $x,y,p\in X$ and $R>0$ :
$$
\begin{cases} \|x-p\|\leq R,\\
\|y-p\|\leq R,\\
x\neq y
\end{cases}\Rightarrow \|\frac{x+y}{2}-p\|<R.
$$
\end{definition}

The proximal pair of the pair $(A,B)$ is denoted by $(A_0,B_0)$ and given by
$$
A_0=\{ x\in A\colon \| x-y'\|={\rm dist}(A,B) \text{ for some $y'\in B$}\},
$$
$$
B_0=\{ y\in B\colon \| x'-y\|={\rm dist}(A,B) \text{ for some $x'\in A$}\}.
$$
The pair $(A,B)$ is said to be a \emph{proximinal pair} if $A=A_0$ and $B=B_0$.
\par
We shall also adopt the notation
 \begin{align*}
 \delta_x(A)&=\sup\{d(x,y)\colon y\in
A\}\,\,\text{for all}\,\, x\in X,\\
\delta(A,B)&=\sup\{d(x,y)\colon x\in A ,\, y\in B\},\\
{\rm diam}(A) &=\delta(A,A).
\end{align*}

\begin{definition}\label{D11}{\rm (\cite{EKV})}
A convex pair $(K_1,K_2)$ in a Banach space $X$ is said to have a proximal normal structure {\rm (PNS)} if
for any bounded, closed, convex and proximinal pair $(H_1,H_2)\subseteq(K_1,K_2)$ for which
${\rm dist}(H_1,H_2)={\rm dist}(K_1,K_2)$ and $\delta(H_1,H_2)>{\rm dist}(H_1,H_2)$, there exists
$(x_1,x_2)\in H_1\times H_2$ such that
$$
\max\{\delta_{x_1}(H_2), \delta_{x_2}(H_1)\}<\delta(H_1,H_2).
$$
\end{definition}

It was announced in \cite{EKV} that every nonempty, bounded, closed and convex pair in a uniformly convex Banach space $X$ has PNS.
\par
Here, we state the following two existence results which are the main conclusions of \cite{EKV}.

\begin{theorem}\label{T11}{\rm (Theorem 2.1 of \cite{EKV})}
Let $(A,B)$ be a nonempty, weakly compact and convex pair in a Banach space $X$ and suppose $(A,B)$ has {\rm PNS}.
Let $T:A\cup B\to A\cup B$ be a cyclic relatively nonexpansive mapping. Then $T$ has a best proximity point.
\end{theorem}

\begin{theorem}\label{T12}{\rm (Theorem 2.2 of \cite{EKV})}
Let $(A,B)$ be a nonempty, weakly compact and convex pair in a strictly convex Banach space $X$ and suppose $(A,B)$ has PNS.
Let $T:A\cup B\to A\cup B$ be a noncyclic relatively nonexpansive mapping. Then $T$ has a best proximity pair.
\end{theorem}

In 2006 the next existence, uniqueness and convergence result of a best proximity point for cyclic contractions was established.

\begin{theorem}\label{T13}{\rm (Theorem 3.10 of \cite{EV})}
Let $(A,B)$ be a nonempty, closed and convex pair in a uniformly
convex Banach space $X$ and let $T:A\cup B\rightarrow A\cup B$ be
a cyclic contraction map. For $x_0\in A$, define $x_{n+1}:=Tx_n$
for each $n\geq 0$. Then there exists a unique $x^\star\in A$ such that
$x_{2n}\rightarrow x^\star$ and $\|x^\star-Tx^\star\|={\rm dist}(A,B)$.
\end{theorem}

Just recently, the noncyclic version of Theorem \ref{T13} was proved in \cite{G2}. Before stating that we recall the following requirements.
\par
For a nonempty subset $A$ of $X$ a \emph{metric projection operator} $\mathcal{P}_A:X\to 2^A$ is defined as
$$
\mathcal{P}_A(x):=\{y\in A : \|x-y\|={\rm dist}(\{x\},A)\},
$$
where $2^A$ denotes the set of all subsets of $A$. It is well known that if $A$ is a nonempty, closed and convex subset of a reflexive and
strictly convex Banach space $X$, then the metric projection $\mathcal{P}_A$ is single valued from $X$ to $A$, that is,
$\mathcal{P}_A:X\to A$ is a mapping with $\|x-\mathcal{P}_A(x)\|={\rm dist}(\{x\},A)$ for any $x\in X$.

\begin{proposition}{\rm (\cite{G1, G2})}\label{P11}
Let $(A,B)$ be a nonempty, bounded, closed and convex pair in a reflexive and strictly convex Banach space $X$. Define $\mathcal{P}:A_0\cup B_0\to A_0\cup B_0$
as
\begin{align}
\mathcal{P}(x)=\begin{cases}\mathcal{P}_{A_0}(x) \quad \text{if} \ x\in B_0,\\
\mathcal{P}_{B_0}(x) \quad \text{if} \ x\in A_0.\end{cases}
\end{align}
Then the following statements hold.\\
$(1)$ \ $\mathcal{P}$ is cyclic on $A_0\cup B_0$ and $\|x-\mathcal{P}x\|={\rm dist}(A,B)$ for any $x\in A_0\cup B_0$,\\
$(2)$ \ $\mathcal{P}$ is an isometry, that is, $\|\mathcal{P}x-\mathcal{P}y\|=\|x-y\|$ for all $(x,y)\in A_0\times B_0$,\\
$(3)$ \ $\mathcal{P}$ is affine,\\
$(4)$ \ $\mathcal{P}^2|_{A_0}=i_{A_0}$ and $\mathcal{P}^2|_{B_0}=i_{B_0}$,\\
$(5)$ \ $\mathcal{P}|_{A_0}$ and $\mathcal{P}|_{B_0}$ are continuous.
\end{proposition}

We are now ready to state a main result of \cite{G2}.

\begin{theorem}\label{T14}{\rm (Theorem 3.2 of \cite{G2})}
Let $(A,B)$ be a nonempty, closed and convex pair in a uniformly convex Banach space $X$ and $T$ a noncyclic contraction
mapping defined on $A\cup B$. Suppose $x_0\in A_0$ and define
\begin{align}
\begin{cases}x_n=T^nx_0,\\ y_n=\mathcal{P}x_n, \end{cases}
\end{align}
for all $n\in\mathbb{N}$, where $\mathcal{P}$ is the projection operator defined in (1). Then the sequence $\{(x_n,y_n)\}\subseteq A_0\times B_0$ converges to
a best proximity pair of the mapping $T$.
\end{theorem}

The main purpose of this paper is to show that the existence of best proximity points for cyclic relatively nonexpansive mappings
is equivalent to the existence of best proximity pairs for noncyclic relatively nonexpansive mappings in the setting of strictly convex Banach spaces.
Then we conclude that Theorem \ref{T12} is a straightforward consequence of Theorem \ref{T11}.
We also obtain a stronger version of Theorem \ref{T14} by using Theorem \ref{T13}.

\section{Main results}

We begin our main conclusions with the following theorem.

\begin{theorem}\label{T21}
Let $(A,B)$ be a nonempty, weakly compact and convex pair in a strictly convex Banach space $X$. Then every cyclic relatively nonexpansive
mapping defined on $A\cup B$ has a best proximity point if and only if every noncyclic relatively nonexpansive mapping defined on
$A\cup B$ has a best proximity pair.
\end{theorem}

\begin{proof}
Assume that every cyclic relatively nonexpansive
mapping defined on $A\cup B$ has a best proximity point and $T:A\cup B\to A\cup B$ is a noncyclic relatively nonexpansive mapping.
If $x\in A_0$,
then there exists a point $y\in B_0$ for which $\|x-y\|={\rm dist}(A,B)$. By the fact that $T$ is a noncyclic relatively nonexpansive mapping,
we obtain $\|Tx-Ty\|={\rm dist}(A,B)$ and so, $Tx\in A_0$ which ensures that $T(A_0)\subseteq A_0$. Similarly, $T(B_0)\subseteq B_0$, that is,
$T$ is noncyclic on $A_0\cup B_0$.
Consider the projection operator $\mathcal{P}$ as in (1). It follows from the proof of Theorem 3.2 of \cite{G2} that
$T$ and $\mathcal{P}$ commute on $A_0\cup B_0$.
Since $\mathcal{P}$ is cyclic and $T$ is noncyclic on $A_0\cup B_0$, we obtain
$$
T\mathcal{P}(A_0)\subseteq T(B_0)\subseteq B_0, \quad T\mathcal{P}(B_0)\subseteq T(A_0)\subseteq A_0.
$$
Therefore, $T\mathcal{P}$ is cyclic on $A_0\cup B_0$.
In view of the fact that $\mathcal{P}$ is an isometry,
$$
\|T\mathcal{P}x-T\mathcal{P}y\|=\|\mathcal{P}Tx-\mathcal{P}Ty\|=\|Tx-Ty\|\leq\|x-y\|,
$$
for all $(x,y)\in A_0\times B_0$, that is, $T\mathcal{P}$ is a cyclic relatively nonexpansive mapping on $A_0\cup B_0$.
Now by assumption, there exists a point $x^\star\in A_0$
such that $\|x^\star-T\mathcal{P}x^\star\|={\rm dist}(A,B)$. Moreover,
$$
\|Tx^\star-T\mathcal{P}x^\star\|=\|Tx^\star-\mathcal{P}Tx^\star\|={\rm dist}(A,B).
$$
Strict convexity of $X$ implies that $Tx^\star=x^\star$. Furthermore,
$$
\mathcal{P}x^\star=\mathcal{P}Tx^\star=T\mathcal{P}x^\star.
$$
Hence, $(x^\star,\mathcal{P}x^\star)$ is a best proximity pair of the mapping $T$.
Conversely, assume that any noncyclic relatively nonexpansive mapping defined on $A_0\cup B_0$ has a best proximity pair
and $S:A\cup B\to A\cup B$ is a cyclic relatively nonexpansive mapping. Thus
$$
S\mathcal{P}(A_0)\subseteq S(B_0)\subseteq A_0, \quad  S\mathcal{P}(B_0)\subseteq S(A_0)\subseteq B_0,
$$
that is, $S\mathcal{P}$ is noncyclic on $A_0\cup B_0$ and again since $\mathcal{P}$ is an isometry, we conclude that $S\mathcal{P}$
is a noncyclic relatively nonexpansive mapping. Now by the assumption $S\mathcal{P}$ has a best proximity pair, called $(p,q)\in A_0\times B_0$.
Thereby,
$$
S\mathcal{P}p=p, \quad S\mathcal{P}q=q, \quad \|p-q\|={\rm dist}(A,B).
$$
We have
$$
\|p-Sp\|=\|S\mathcal{P}p-Sp\|\leq\|\mathcal{P}p-p\|={\rm dist}(A,B),
$$
and so $p$ is a best proximity point of $S$. Similarly, we can see that $q$ is a best proximity point of $S$ in $B_0$ and this completes the proof.
\end{proof}

\begin{corollary}\label{C21}
Theorem \ref{T12} is a consequence of Theorem \ref{T11} (see the proof of Theorem \ref{T12} in \cite{EKV}).
\end{corollary}

Here, we compare Theorem \ref{T13} with Theorem \ref{T14}.

\begin{theorem}\label{T22}
Theorem \ref{T14} implies Theorem \ref{T13} when the initial point of the iterative sequence in Theorem \ref{T13} is chosen in $A_0$.
\end{theorem}

\begin{proof}
Let $(A,B)$ be a nonempty, closed and convex pair in a uniformly convex Banach space $X$ and $S:A\cup B\to A\cup B$ be a cyclic contraction.
From Proposition 3.3 of \cite{EV} the pair $(A_0,B_0)$ is nonempty. Consider the mapping $S\mathcal{P}$ on $A_0\cup B_0$.
A proof which is similar to the one of Theorem \ref{T21} shows that
the mapping $S\mathcal{P}$ is a noncyclic contraction on $A_0\cup B_0$.
Thus for any $x_0\in A_0$ if we define
\begin{align}
\begin{cases}x_n=(S\mathcal{P})^nx_0,\\ y_n=\mathcal{P}x_n, \end{cases}
\end{align}
then $\{(x_n,y_n)\}$ converges to a best proximity pair of the mapping $S\mathcal{P}$. Suppose
$x_n\to p$. Continuity of the projection operator $\mathcal{P}$ implies that $y_n\to\mathcal{P}p$. Since $\mathcal{P}^2$ is identity on $A_0$
and $B_0$, respectively (Proposition \ref{P11}),
$$
S\mathcal{P}p=p, \quad \mathcal{P}p=S\mathcal{P}(\mathcal{P}p)=Sp,
$$
and
$$
\|p-Sp\|=\|p-\mathcal{P}p\|={\rm dist}(A,B).
$$
Hence, $p\in A_0$ is a best proximity point for the mapping $S$. We also note that for any $n\in\mathbb{N}$
$$
x_{2n}=(S\mathcal{P})^{2n}x_0=S^{2n}\mathcal{P}^{2n}x_0 \ (\text{since} \ S \ \text{and} \ \mathcal{P} \ \text{commute})
$$
$$
=S^{2n}x_0\to p.
$$

\end{proof}

\begin{theorem}
Theorem \ref{T13} implies Theorem \ref{T14} for an even subsequence of the iterative sequence defined in (2).
\end{theorem}

\begin{proof}
Let $(A,B)$ be a nonempty, closed and convex pair in a uniformly convex Banach space $X$ and $T:A\cup B\to A\cup B$ be a noncyclic contraction.
From Proposition 3.4 of \cite{FG} the pair $(A_0,B_0)$ is nonempty. Consider the mapping $T\mathcal{P}$ on $A_0\cup B_0$.
The mapping $T\mathcal{P}$ is a cyclic contraction on $A_0\cup B_0$.
Thus for all $x_0\in A_0$ if we define $x_n=(T\mathcal{P})^nx_0$, then the sequence $\{x_{2n}\}$ converges to a best proximity point
of $T\mathcal{P}$, say $x^\star\in A_0$. In view of the fact that $T\mathcal{P}$ is a cyclic relatively nonexpansive mapping by Theorem \ref{T21}
we conclude that $(x^\star,\mathcal{P}x^\star)$ is a best proximity pair for $T$.
From Proposition \ref{P11}, since $\mathcal{P}^2$ is the identity map on $A_0$, we must have
$$
x_{2n}=(T\mathcal{P})^{2n}x_0=T^{2n}\mathcal{P}^{2n}x_0=T^{2n}x_0\to x^\star.
$$
Continuity of the projection operator on $A_0$ ensures that $y_{2n}:=\mathcal{P}x_{2n}\to \mathcal{P}x^\star$ and so the sequence
$\{(x_{2n},\mathcal{P}x_{2n})\}$ converges to a best proximity pair of $T$.
Notice that for any $n\in\mathbb{N}$ we have
$$
x_{2n+1}=(T\mathcal{P})^{2n+1}x_0=T^{2n+1}\mathcal{P}^{2n+1}x_0=T^{2n+1}\mathcal{P}x_0\in B_0.
$$

\end{proof}

\begin{corollary}
The convergence results of Theorem \ref{T13} and Theorem \ref{T14} are independent. That is, the convergence result of Theorem \ref{T13}
cannot be implied by the convergence result of Theorem \ref{T14} and vice versa.
\end{corollary}

\bigskip
{\em Acknowledgement}: The second author would like to thank the National Research Foundation of South Africa for partial financial support (Grant Number: 118517).


\begin{thebibliography}{99}



\bibitem{EKV}
A.A. Eldred, W.A. Kirk, and P. Veeramani, \emph{Proximal normal
structure and relatively nonexpansive mappings}, Studia Math.,
\textbf{171}~(2005), 283--293.


\bibitem{EV}
A.A. Eldred, P. Veeramani, \emph{Existence and convergence
of best proximity points}, J. Math. Anal. Appl,
\textbf{323}, (2006), 1001--1006.



\bibitem{FG}
A. Fern\'andez-Le\'on, M. Gabeleh, \emph{Best proximity pair theorems for noncyclic mappings in Banach and metric spaces}, Fixed Point Theory,
\textbf{17}~(2016), 63--84.



\bibitem{G1}
M. Gabeleh, \emph{Common best proximity pairs in
strictly convex Banach spaces}, Georgian Math. J., \textbf{24}~(2017), 363--372.


\bibitem{G2}
M. Gabeleh, \emph{Convergence of Picard's iteration using projection algorithm for noncyclic contractions}, Indag. Math., \textbf{30}~(2019), 227--239.




\end{thebibliography}
\end{document}